\documentclass[11pt,draft]{amsart}
\usepackage{graphicx}
\usepackage{amssymb}
\usepackage{amsmath}
\usepackage{amsthm}
\usepackage{amscd}
\usepackage{epsfig}
\usepackage[pagewise]{lineno}
\makeatletter
\@namedef{subjclassname@2020}{%
\textup{2020} Mathematics Subject Classification} \makeatother
\usepackage[colorlinks=true]{hyperref}
\usepackage{amsfonts}
\usepackage[top=30mm, bottom=35mm, left=25mm, right=25mm]{geometry}

\newtheorem{theorem}{Theorem}[section]

\newtheorem{corollary}[theorem]{Corollary}
\newtheorem{lemma}[theorem]{Lemma}

\newtheorem{remark}[theorem]{Remark}
\theoremstyle{definition}
\newtheorem{defin}[theorem]{Definition}
\numberwithin{equation}{section}

\newcommand{\lk}{\left}
\newcommand{\re}{\right}
\begin{document}

\title{Directional dynamics of $\mathbb{Z}_+\times \mathbb{Z}$-actions
generated by 1D-CA and the shift map}

\author [H. Ak\i n  C. Liu]{Hasan Ak\i n and Chunlin Liu}

\address{H. Ak\i n: International Centre for Theoretical Physics
(ICTP), Strada Costiera, 11 I - 34151 Trieste Italy}
\email{hakin@ictp.it, akinhasan25@gmail.com}

\address{C. Liu: CAS Wu Wen-Tsun Key Laboratory of Mathematics,
School of Mathematical Sciences,
University of Science and Technology of China, Hefei, Anhui, 230026, PR China}
\email{lcl666@mail.ustc.edu.cn}

\subjclass[2020]{Primary: 37A25, 37A35, 28D05, 28D20, 37A44}

\keywords {Directional sequence entropy, cellular automata, directional weak
mixing}
\date{\today }%
\begin{abstract}
In this short paper, we compute the directional sequence entropy
for of $\mathbb{Z}_+\times \mathbb{Z}$-actions generated by
cellular automata and the shift map. Meanwhile, we study the
directional dynamics of this system. As a corollary, we prove that
there exists a sequence  such that for any direction, some of the
systems above have positive  directional sequence entropy.
Moreover,  with help of mean ergodic theory for directional weak
mixing systems, we obtain a  result of number theory about
combinatorial numbers.
\end{abstract}

\maketitle

\markboth{Directional dynamics of $\mathbb{Z}_+\times
\mathbb{Z}$-actions}{H. Ak\i n and C. Liu}

\section{Introduction}
In the late 1940s, cellular automata (CA) were created by John von
Neumann \cite{Neumann}, who was inspired by biological
applications. A cellular automaton is a sort of dynamical system
invented by Ulam \cite{Ulam} and von Neumann \cite{Neumann} as a
model for self-production. 1-dimensional CA (1D-CA) is made up of
an infinite lattice with finite states and a mapping called the
local rule. Hedlund \cite{Hedlund} takes a methodical approach to
CA from a mathematical standpoint. Since 1940, CA have continued
to be the focus of extreme attention by many researchers
\cite{Akin-2005b}.

The entropy of a system has been studied extensively for various
purposes in fields such as computer science, mathematics, physics,
chemistry, information theory. It is well known that the entropy
measures the chaoticity or unpredictability of a system. In
ergodic theory, there exist numerous notions of entropy of the
measure-preserving transformation on the probability space (e.g.,
measure entropy, topological entropy, directional entropy,
rotational entropy, etc.) \cite{Ward,Akin-2003,Akin-2005a}. Ward
\cite{Ward} investigated the topological entropy of 1-dimensional
linear CA (1D-LCA). If the local rule $f$ is defined by
$f(x_{-k},\cdots,x_k)=\sum_{i=-k}^{k}x_i\ (\text{mod}\ m)$ for
$k\in \mathbb{N}$, the uniform Bernoulli measure, as proven by
Ak\i n \cite{Akin-2003}, is a measure with maximum entropy. In
Ref. \cite{Akin-2005a}, the author calculated the
measure-theoretic entropy and directional entropy of
$\mathbb{Z}\times \mathbb{N}$-actions generated by some linear
1D-LCA and the shift map. In Ref. \cite{Akin-2009}, the author
derived a formula to compute the topological directional entropy
of $\mathbb{Z}^2$-action using the coefficients of the local rule.
Ak\i n et al. \cite{ABC-2013} studied the quantitative behavior of
1-D linear cellular automata (LCAs) and proved that the Hausdorff
of the limit set of a LCA is the unique root of Bowen's equation.
Chang, and Ak\i n \cite{CA-2016} proved that every invertible
1D-LCA is a Bernoulli automorphism without making use of the
natural extension.

Since Milnor \cite{Milnor-1988} introduced a notion of directional
dynamics, the study of directional dynamics has led to other
productive lines of research, the most notable among these being
Boyle and Lind's work on expansive sub-dynamics \cite{BL}. In
\cite{Ai} Johnson and \c Sahin studied  directional recurrence
properties for $\mathbb{Z}^d$-actions. Using classical sequence
entropy \cite{K,KN,Saleski} and directional entropy
\cite{Akin-2005a,Milnor-1988}, Liu and Xu \cite{LX} introduced
directional sequence entropy for $\mathbb{Z}^q$-actions. Moreover,
they \cite{LX,LX2} have defined directional properties
intrinsically and studied their relationship to spectrum.
Dennunzio et al. \cite{Dennunzio-2009} provided a circumstantial
classification for a class of the LCA by taking into account
non-trivial examples to study some of Sablik's categorization
classes and extended studies of directional dynamics to include
factor languages and attractors.

In the present paper, we consider $\mathbb{Z}_+\times
\mathbb{Z}$-actions generated by 1D-CA and the shift map on space
$\mathbb{Z}^{\mathbb{Z}}_{a}$ consisting two-sided infinite
sequences, where $\mathbb{Z}_{a}$ is a finite ring. In Section 2,
we recall some notions and results that we will use. In Section 3,
we compute the directional sequence entropy for
$\mathbb{Z}_+\times \mathbb{Z}$-actions. Meanwhile, we study the
directional dynamics behavior of this system. We prove that
$\mathbb{Z}_+\times \mathbb{Z}$-systems are directional weak
mixing along any direction $\vec{v}\in S^1$, under given
conditions in Section 4. Moreover, we take advantage of
directional version of mean ergodic theory introduced by the
second author \cite{L} to obtain a number theory result about
combinatorial numbers. That is, for $m$-a.e. $x\in(0,1)$, the
probability that the number $0$ or $1$ appears in the sequence
$\{\sum_{l=0}^{n}\binom{n}{l}x_{l+1} (\bmod 2)\}_{n=1}^\infty$ is
$1/2$, where $\binom{n}{l}$ is the combinatorial number,
$0.x_1x_2\ldots$ is the $2$-adic development of $x$ and $m$ is the
Lebesgue measure on $[0,1]$. In fact, we prove a stronger result
in Section 5.

\section{Preliminaries}
\subsection{Cellular automaton}
Let $\mathbb{Z}_{a}=\{0,1,\cdots,a-1\}$ be a finite ring. The
compact topological space $\mathbb{Z}^{\mathbb{Z}}_{a}$ consists
of two-sided infinite sequences denoted as
$x=(x_n)_{n=-\infty}^{\infty},$ where $x_n\in \mathbb{Z}_{a}.$ Let
$T_{f[-r,r]}:\mathbb{Z}_{a}^{\mathbb{Z}}\rightarrow
\mathbb{Z}_{a}^{ \mathbb{Z}}$  be a map that acts locally on the
set of two-sided infinite sequences, where
$f:\mathbb{Z}_{a}^{2r+1}\rightarrow \mathbb{Z}_{a}$ is called
local rule.
\begin{defin}
Let us define the local rule $f$ by
\begin{equation}\label{LR1}
f(x_{-r},\cdots,x_{r})=\overset{r}{\underset{i=-r}{\sum }}\lambda
_{i}x_{i}\ (\text{mod}\ a), \end{equation}
 where at least one of $\lambda_{-r},\cdots ,\lambda _{r}$ is
 nonzero. The cellular automaton $T_{f[-r,r]}$ generated by $f$ given in
 \eqref{LR1} is defined as;
\begin{equation}
(T_{f[-r,r]}(x)):=(y_n)_{-\infty}^{\infty},y_n=f(x_{n-r},
\cdots,x_{n+r})=\overset{r}{\underset{i=-r}{\sum
}}\lambda _{i}x_{n+i}\ (\text{mod}\ a),
\end{equation}
where the positive integer $r$ is called the radius of the local
rule $f$.
\end{defin}
We consider the semigroup $\mathbb{Z}_+\times \mathbb{Z}$-action
$\Phi$ defined by
\begin{equation}\label{action}
\Phi^{(m,n)}:=T^{m}_{f[-r,r]}\circ \sigma^{n},\ m\in \mathbb{Z}_+\
\text{and}\ n\in \mathbb{Z},
\end{equation}
where $\sigma$ is the shift map from $\mathbb{Z}_a^{\mathbb{Z}}$
to $\mathbb{Z}_a^{\mathbb{Z}}$ given by $(\sigma(x))_i=x_{i+1}$
for $i\in \mathbb{Z}$, $x\in \mathbb{Z}_a^{\mathbb{Z}}$. Note that
the shift map $\sigma$ is one of the simplest examples of CA.

\subsection{Directional sequence entropy and directional weak mixing}
To study directional system, Liu and Xu \cite{LX} introduced a new
invariant directional sequence entropy. We recall it as follows.

Let $(\mathbb{Z}_a^\mathbb{Z},\mathcal{X},\mu, \Phi)$ be a
$\mathbb{Z}^2$-measure preserving system ($\mathbb{Z}^2$-m.p.s.
for short) and $\vec{v}=(1,\beta)\in \mathbb{R}^2$ be a direction
vector, where $\Phi$ is a $\mathbb{Z}_+\times \mathbb{Z}$-action
given in \eqref{action}, $\mathcal{X}$ is a $\sigma$-algebra
generated by the cylinder sets and $\mu$ is a product measure on
$\mathbb{Z}_a^\mathbb{Z}$ (see \cite{Denker} for details). For the
sake of simplicity, we write the vector $\vec{v}$ by $(1,\beta)$.
Someone can show that all results in this paper are true for
$\vec{v}=(0,1)$ since this is reduced to the special case of
$\mathbb{Z}$-actions. We put
$$
\Lambda^{\vec{v}}(b)=\left\{(m,n)\in\mathbb{Z}^2:\beta m-b/2\leq
n\leq \beta m+b/2\right\}
$$
and write $\Lambda_k^{\vec{v}}(b)=\Lambda^{\vec{v}}(b)\cap
([0,k-1]\times \mathbb{Z})$.
\subsubsection{Directional sequence entropy.}
For a finite measurable partition
$\alpha$ of $\mathbb{Z}_a^\mathbb{Z}$, let
$$H_{\mu}(\alpha)=-\sum_{A\in \alpha}\mu(A)\log{\mu(A)}.$$
Let us consider any infinite subset
$S=\{(m_i,n_i)\}_{i=1}^{\infty}$ of $\Lambda^{\vec{v}}(b)$ such
that $\{m_i\}_{i=1}^{\infty}$ is strictly monotone, from the
definition of sequence entropy \cite{K,KN,Saleski} we obtain
$$
h^S_{\mu}(\Phi,\alpha)= \limsup_{k\to \infty} \frac{1}{k}H_\mu
\lk(\bigvee_{i=1}^k \Phi^{-(m_i,n_i)} \alpha\re).
$$
Then one can define the directional sequence entropy of $\Phi$ for
the subset $S$ by
$$
h^S_{\mu}(\Phi)=\sup_{\alpha}h^S_{\mu}(\Phi,\alpha),
$$
where the supremum is taken over all finite measurable partitions
of $\mathbb{Z}_a^\mathbb{Z}$.
\subsubsection{Directional weak mixing}
Let $\mathcal{A}_c^{\vec{v}}(b)$ be the collection
of $f\in \mathcal{H}:=L^2(\mathbb{Z}_a^\mathbb{Z},\mathcal{X},\mu)$ such that
$$
\overline{\left\{U_\Phi^{(m,n)}f:(m,n)\in \Lambda^{\vec{v}}(b)
\right\}}\text{ is compact in
}L^2(\mathbb{Z}_a^\mathbb{Z},\mathcal{X},\mu),$$ where
$U_\Phi^{(m,n)}:L^2(\mathbb{Z}_a^\mathbb{Z},\mathcal{X},\mu)\to
L^2(\mathbb{Z}_a^\mathbb{Z},\mathcal{X},\mu)$ is the unitary
operator such that
\begin{equation*}
U_\Phi^{(m,n)}f=f\circ \Phi^{(m,n)}\text{ for all }f\in L^2(\mathbb{Z}_a^\mathbb{Z},\mathcal{X},\mu).
\end{equation*}
One can easily show that $\mathcal{A}_c^{\vec{v}}(b)$ is a
$U_{\Phi^{\vec{w}}}$-invariant for any vector $\vec{w}\in
\mathbb{Z}^2$ and conjugation-invariant subalgebra of
$\mathcal{H}$. Then there exists a $\Phi$-invariant
sub-$\sigma$-algebra $\mathcal{K}_\mu^{\vec{v}}(b)$  of
$\mathcal{X}$ such that
\begin{align}
\label{1}\mathcal{A}_c^{\vec{v}}(b)=L^2(\mathbb{Z}_a^\mathbb{Z},\mathcal{K}_\mu^{\vec{v}}(b),\mu).
\end{align}
Directly from \eqref{1}, we define the $\vec{v}$-directional
Kronecker algebra of $(\mathbb{Z}_a^\mathbb{Z},\mathcal{X},\mu,
\Phi)$ by
$$
\mathcal{K}_\mu^{\vec{v}}=\left\{B\in\mathcal{X}:
\overline{\left\{U_\Phi^{(m,n)}1_B :(m,n)\in \Lambda^{\vec{v}}(b)
\right\}}\text{ is compact in }
L^2(\mathbb{Z}_a^\mathbb{Z},\mathcal{X},\mu) \right\}.
$$
Note that the definition of $\mathcal{K}_\mu^{\vec{v}}(b)$ is
independent of the selection of $b\in (0,\infty)$ (see
\cite[Proposition 3.1]{LX}). So we omit $b$ in
$\mathcal{K}_\mu^{\vec{v}}(b)$ and write it as
$\mathcal{K}_\mu^{\vec{v}}$. We say $\mu$ is $\vec{v}$-weak mixing
if $\mathcal{K}_\mu^{\vec{v}}=\{X,\emptyset\}$.
\section{Computation of directional sequence entropy}
In this section, we will compute the sequence entropy and
directional sequence entropy of $\mathbb{Z}_+\times
\mathbb{Z}$-actions obtained by 1D-CA and the shift map. We need
the following lemma (see \cite{K} for $\mathbb{Z}$-actions).
\begin{lemma}\label{lem1}
Let  $(\mathbb{Z}_a^\mathbb{Z},\mathcal{X},\mu,\Phi)$ be a
$\mathbb{Z}^2$-m.p.s., where $\Phi$ is a function given in
\eqref{action}, $\left\{\xi_{k}\right\}$ be a sequence of
measurable partitions such that $\xi_{1} \leq \xi_{2} \leq \cdots$
$\leq \xi_{n} \leq \cdots$, and $\bigvee_{i=1}^{\infty}
\xi_{i}=\epsilon$ (where $\epsilon$ denotes the partition into
points of $(\mathbb{Z}_a^\mathbb{Z}, \mu)$ ). Then for any $S$ and
$\Phi, h^{S}(\Phi)=\lim _{k \rightarrow \infty} h^S\left(\Phi,
\xi_{k}\right)$.
\end{lemma}

\begin{theorem}\label{Thm1}
Let $\mu$ be the uniform Bernoulli measure on $\mathbb{Z}_a^\mathbb{Z}$,
$$
f\left(x_{n-r},\cdots,x_{n+r}\right)=\sum\limits_{i=-r}^{r}\lambda_i x_{n+i}(\bmod\ a)
$$
with $gcd(\lambda_{-r},a)=1$ and $gcd(\lambda_{r},a)=1$, and
$\Phi$ be a $\mathbb{Z}_+\times \mathbb{Z}$-action defined in
\eqref{action}. Then for the $\mathbb{Z}_+\times
\mathbb{Z}$-m.p.s.
$(\mathbb{Z}_a^\mathbb{Z},\mathcal{X},\mu,\Phi)$ and
$S=\{(m_i,n_i)\}_{i=1}^\infty\subset\mathbb{Z}_+\times\mathbb{Z}$
such that $\{m_i\}_{i=1}^\infty$ is a strictly monotone increasing
syndetic set with gap $N\in\mathbb{N}$ and $m_i>n_i$, one has
\[h^{S}(\Phi)=2r\log a\cdot\limsup_{l\to\infty}\frac{m_l}{l}.\]
\begin{proof}
If we choose $M\in\mathbb{N}$ large enough then $\xi(-M, M)
\vee\Phi^{-(m_1,n_1)} \xi(-M, M) $ consists of all cylinder sets
in the form
\[_{-\left(r m_1+M\right)-n_1}\left[j_{-r m_1-M},
\ldots, j_{r m_1+M}\right]_{\left(r m_1+M\right)-n_1}.
\]
Since $2(r m_1+M)+1>2r m_2+1$ it follows that
$$\xi(-M, M)
\vee\Phi^{-(m_1,n_1)} \xi(-M, M)\vee\Phi^{-(m_2,n_2)} \xi(-M, M)
$$
consists of all cylinder sets in the form
\[_{-\left(r m_2+M\right)-n_2}\left[j_{-r m_2-M}, \ldots,
j_{r m_2+M}\right]_{\left(r m_2+M\right)-n_2}.\]

By the same way we can prove that  $
\bigvee_{i=0}^{l}\Phi^{-(m_i,n_i)} \xi(-M, M) $ consists of all
cylinder sets in the form
\[
_{-\left(r m_l+M\right)-n_l}\left[j_{-r m_l-M}, \ldots,
j_{r m_l+M}\right]_{\left(r m_l+M\right)-n_l}.
\]
Thus, we get
\begin{equation*}
\begin{split}
&h^{S}(\Phi,\xi(-M,M))=\limsup_{l\to\infty}\frac{1}{l+1}H_{\mu}(\bigvee_{i=0}^{l}\Phi^{-(m_i,n_i)}\xi(-M,M))\\
=&\limsup_{l\to\infty}\frac{1}{l+1}a^{2(r m_l+M)-1}
\mu\lk(_{-\left(r m_l+M\right)-n_l}\left[j_{-r m_l-M}, \ldots, j_{r m_l+M}\right]_{\left(r m_l+M\right)-n_l}\re)\\
\quad&\times \log\mu\lk(_{-\left(r m_l+M\right)-n_l}\left[j_{-r m_l-M}, \ldots, j_{r m_l+M}\right]_{\left(r m_l+M\right)-n_l}\re) \\
=&\limsup_{l\to\infty}-\frac{1}{l+1}\log a^{-2(r m_l+M)-1}
=\lk(\limsup_{l\to\infty}\frac{2(r m_l+M)+1}{l+1}\re)\cdot\log a\\
=&2r \log a\cdot\limsup_{l\to\infty}\frac{m_l}{l}.
\end{split}
\end{equation*}
By Lemma \ref{lem1} and the fact that $\bigvee_{i=M}^{\infty}
\xi(-i,i)=\epsilon$, one has
\[
h^{S}(\Phi)=2r\log a\cdot\limsup_{l\to\infty}\frac{m_l}{l}.
\]
This finishes the proof.
\end{proof}
\end{theorem}
\begin{corollary}
For the $\mathbb{Z}_+\times \mathbb{Z}$-m.p.s.
$(\mathbb{Z}_a^\mathbb{Z},\mathcal{X},\mu,\Phi)$ defined in
Theorem \ref{Thm1}, if $\vec{v}=(x,y)\in S^1$ with $x>y$, then the
directional sequence entropy is only depends on the gap of the
first coordinate, not direction.
\end{corollary}

\section{Directional weak mixing}
In this section, we are going to prove several and important new
results.  To study the complexity of directional systems, Liu
\cite{L} introduced directional weak mixing. Since this paper
considers $\mathbb{Z}_+\times \mathbb{Z}$-actions obtained by
cellular automata and the shift map, we only recall results in
\cite{L} for it. In fact, the following results holds for all
$\mathbb{Z}^2$-actions.

Now we will prove a class of systems with $\mathbb{Z}_+\times
\mathbb{Z}$-actions generated by 1D-CA and the shift map is directional weak
mixing along any direction $\vec{v}\in S^1.$ To prove this result
let us begin the following lemma \cite{L}.
\begin{lemma}\label{lemma1}
Let  $(\mathbb{Z}_a^\mathbb{Z},\mathcal{X},\mu,\Phi)$ be a $\mathbb{Z}^2$-m.p.s. and
$\vec{v}=(1,\beta)\in\mathbb{R}^2$ be a direction vector. Then the
following statements are equivalent.
\begin{itemize}
\item [(a)] $(\mathbb{Z}_a^\mathbb{Z},\mathcal{X},\mu,\Phi)$ is $\vec{v}$-weak mixing.
\item [(b)] There exists $b>0$ such that $$\lim_{k\rightarrow
\infty}\frac{1}{\#(\Lambda_k^{\vec{v}}(b))}\sum_{(m,n)\in\Lambda_k^{\vec{v}}(b)}|\langle
U_\Phi^{(m,n)}1_B,1_C\rangle-\mu(B)\mu(C)|=0,$$ for any $B,C\in
\mathcal{X}$.
\item [(c)]  For any $b>0$,
$$
\lim_{k\rightarrow
\infty}\frac{1}{\#(\Lambda_k^{\vec{v}}(b))}\sum_{(m,n)\in\Lambda_k^{\vec{v}}(b)}|\langle
U_\Phi^{(m,n)}1_B,1_C\rangle-\mu(B)\mu(C)|=0,
$$
for any $B,C\in
\mathcal{X}$.
\end{itemize}
\end{lemma}
Considering lemma \ref{lemma1}, we have the following result.
\begin{theorem}\label{thm1}
Let $\mu$ be the uniform Bernoulli measure on
$\mathbb{Z}_a^\mathbb{Z}$,
$$
f\left(x_{0},\cdots,x_{n+r}\right)=\sum\limits_{i=0}^{r}\lambda_i x_{n+i}(\bmod\ a)
$$
and $\Phi$ be a $\mathbb{Z}_+\times \mathbb{Z}$-action given in
\eqref{action}. Then the system
$(\mathbb{Z}_a^\mathbb{Z},\mathcal{X},\mu,\Phi)$ is directional
weak mixing along any direction $\vec{v}\in S^1.$
\end{theorem}
\begin{proof}
Note that  for any $M,N\in\mathbb{N}$, there exists $n>N+M$ such that for any $m\in\mathbb{Z}_+$, one has
\[
\mu(\Phi^{-(m,n)}\xi(-M,M)\cap\xi(-N,N))=\mu(\Phi^{-(m,n)}\xi(-M,M))\mu(\xi(-N,N)).
\]
We can easily show that the system associated with \eqref{action}
satisfies (b) of Lemma \ref{lemma1} along any $\vec{v}\in S^1$.
Hence we immediately obtain that the above system is directional
weak mixing along any direction $\vec{v}\in S^1.$
\end{proof}
With the help of \cite[Theorem 1.3]{L}, we immediately obtain that
another result about directional sequence entropy for the system
defined in Theorem \ref{thm1}.
\begin{corollary}
Let $(\mathbb{Z}_a^\mathbb{Z},\mathcal{X},\mu,\Phi)$ be the
$\mathbb{Z}^2$-m.p.s which is defined in Theorem \ref{thm1}. For
any direction vector $\vec{v}\in S^1$, there exists a sequence
$S=\{(m_i,n_i)\}_{i=1}^{\infty}$ of $\Lambda^{\vec{v}}(b)$
satisfying equality $h^{S}_{\mu}(\Phi,\alpha)=H_{\mu}(\alpha)$ for
any finite measurable partition $\alpha$ of
$\mathbb{Z}_a^\mathbb{Z}$.
\end{corollary}

\section{A result in number theory}
In this section, we will obtain a result about combinatorial
mathematics. To complete the proof, we need to recall some
results. In \cite{L}, the second author obtained a mean ergodic
theory for directional weak mixing, which is the key to prove our
result, and we restated it  as follows. We remark that the
following results, that is, Theorem \ref{thm22} and Theorem
\ref{thm3} hold for any $\mathbb{Z}^d$-m.p.s.
\begin{theorem}\label{thm22}
Let  $(\mathbb{Z}_a^\mathbb{Z},\mathcal{X},\mu,\Phi)$ be a
$\mathbb{Z}^2$-m.p.s., $\vec{v}=(1,\beta)\in\mathbb{R}^2$ be a
directional vector and $b\in(0,\infty)$. Then the following
statemens are equivalent.
\begin{itemize}
\item[(a)] $(\mathbb{Z}_a^\mathbb{Z},\mathcal{X},\mu,\Phi)$ is
$\vec{v}$-weak mixing.

\item[(b)] For any infinite subset $Q=\{(m_i,n_i)\}_{i=1}^\infty$
of $\Lambda^{\vec{v}}(b)$  with $\liminf_{n \rightarrow
\infty}\frac{\#{(Q\cap\Lambda_k^{\vec{v}}(b))}}{\#(\Lambda_k^{\vec{v}}(b))}>0$,
one has
\begin{equation*}
\lim_{N\to
\infty}\|\frac{1}{N}\sum_{i=1}^{N}U_\Phi^{(m_i,n_i)}g-\int_{\mathbb{Z}_a^\mathbb{Z}}gd\mu\|_2=0
\end{equation*}
for all $g\in L^2(\mathbb{Z}_a^\mathbb{Z},\mathcal{X},\mu)$, where
$\#(A)$ is the number of elements in finite subset $A$.
\end{itemize}
\end{theorem}
%
Moreover, for direction
$\vec{v}=(m,n)\in\mathbb{Z}^2\setminus\{(0,0)\}$, one has the
following version Birkhoff ergodic theory. We remark that the
proof follows methods in \cite{A}.
\begin{theorem}\label{thm3}
Let  $(\mathbb{Z}_a^\mathbb{Z},\mathcal{X},\mu,\Phi)$ be a
$\mathbb{Z}^2$-m.p.s. and $\vec{v}=(m,n)\in\mathbb{Z}^2$ be a
direction vector. If
$(\mathbb{Z}_a^\mathbb{Z},\mathcal{X},\mu,\Phi)$ is $\vec{v}$-weak
mixing then for $\mu$-a.e. $x\in \mathbb{Z}_a^\mathbb{Z}$
\begin{equation*}
\lim_{N\to
    \infty}\frac{1}{N}\sum_{k=0}^{N-1}U_\Phi^{(km,kn)}g(x)=\int_{\mathbb{Z}_a^\mathbb{Z}}gd\mu
\end{equation*}
for all $g\in L^1(\mathbb{Z}_a^\mathbb{Z},\mathcal{X},\mu)$.
\end{theorem}
\begin{proof}

Let $\mathbb{A}_{N} g(x)=\frac{1}{N} \sum_{k=0}^{N-1}
g\left(\Phi^{(km,kn)}x\right)$. Given $\phi \in
L^{1}(\mathbb{Z}_a^\mathbb{Z},\mathcal{X},\mu)$, let
$$
M_{N} \phi=\max \left\{\sum_{j=0}^{k-1} \phi \circ \Phi^{(jm,jn)}: 1 \leq k \leq N\right\}
$$
It is easy to see that  $\mathbb{A}_{N} \phi \leq \frac{1}{N} M_{N} \phi$. Let
$$
A(\phi)=\left\{x \in \mathbb{Z}_a^\mathbb{Z}: \sup _{N} M_{N} \phi(x)=\infty\right\}.
$$
Then is $A(\phi)$ is a $\Phi^{(m,n)}$-invariant set. By Theorem
\ref{thm22}, one has $\mu(A(\phi))=0$ or $1$. Given $g\in
L^1(\mathbb{Z}_a^\mathbb{Z},\mathcal{X},\mu)$ and $\epsilon>0$,
let $\phi=g-\int gd\mu-\epsilon$. Suppose to the contrary that
$\mu(A(\phi))=1$. Then
$$
0<\int_{A(\phi)} \phi d \mu=\int_{A(\phi)}\left(g-\int
gd\mu-\epsilon\right) d \mu=-\varepsilon \mu(A) \leq 0
$$
a contradiction. Hence $\mu(A(\phi))=0$. Note for any $x \in
A(\phi)^{c}$, one has $ \limsup _{n \rightarrow \infty}
\mathbb{A}_{n} \phi(x) \leq 0 . $ So, one gets
$$
\limsup _{n \rightarrow \infty} \mathbb{A}_{n} \phi(x) \leq 0
$$
for $\mu$-a.e. $x\in \mathbb{Z}_a^\mathbb{Z}$. Since
$\mathrm{A}_{n} \phi=\mathrm{A}_{n} (g-\int gd\mu-\epsilon)$, one
has
$$
\limsup _{n \rightarrow \infty} \mathbb{A}_{n} g(x) \leq \int
gd\mu+\epsilon .
$$
Finally, do same thing for $-g$, one has
$$
\liminf _{n \rightarrow \infty} \mathbb{A}_{n} g(x) \geq \int
gd\mu-\epsilon.
$$
As $\epsilon>0$ is arbitrary, we obtain that  for $\mu$-a.e. $x\in \mathbb{Z}_a^\mathbb{Z}$
\begin{equation*}
\lim_{N\to
    \infty}\frac{1}{N}\sum_{k=0}^{N-1}U_\Phi^{(km,kn)}g(x)=\int_{\mathbb{Z}_a^\mathbb{Z}}gd\mu
\end{equation*}
for all $g\in L^1(\mathbb{Z}_a^\mathbb{Z},\mathcal{X},\mu)$.
\end{proof}

With help of Theorem \ref{thm3}, we obtain an interesting result about
number theory.
\begin{theorem}\label{thm2}
Given $N\in\mathbb{Z}$, $k\in\mathbb{N}$ and
$j\in\{0,1,\ldots,k-1\}$, for $m$-a.e. $x\in(0,1)$, the
probability that the number $j$ appears in the sequence
$\{\sum_{l=0}^{nN}\binom{nN}{l}x_{l+1} (\bmod k)\}_{n=1}^\infty$
is $1/k$, where $\binom{nN}{l}$ is the combinatorial number,
$0.x_1x_2\ldots$ is the $k$-adic development of $x$ and $\nu$ is
the Lebesgue measure on $[0,1]$.
\end{theorem}
\begin{proof}
For simplicity, we only prove the case of $k=2$,  because other
cases are similar to prove. Let the local rule $f$ be given as
$f(x_{-1},x_0,x_{1})=x_0+x_1$ ($\bmod\ 2$). We consider the
semigroup $\mathbb{Z}_+\times \mathbb{Z}_+$-action $\Phi$ defined
by $ \sigma^{s}\circ T^{t}_{f[-1,1]}=\Phi^{(s,t)}$, $s,t\in
\mathbb{Z}_+$, where $\sigma$ is the shift map from
$\{0,1\}^{\mathbb{Z}_+}$ to $\{0,1\}^{\mathbb{Z}_+}$. Suppose
$\mu$ is the uniform Bernoulli measure on the product space
$\{0,1\}^{\mathbb{Z}_+}$ and let $\mathcal{X}$ be the
$\sigma$-algebra generated by the cylinder sets. Then we obtain a
$\mathbb{Z}_+^2$-m.p.s.
$(\{0,1\}^{\mathbb{Z}_+},\mathcal{X},\mu,\Phi)$

Let $Y$ be all real numbers with unique $2$-adic development. Then
$\nu(Y)=1$. For any $x\in Y$ with $2$-adic development
$x=0.x_0x_1,\ldots$, we identify it as
$(x_i)_{i=0}^\infty\in\{0,1\}^{\mathbb{Z}_+}$. Let us consider
cylinder set $A={}_0[0]_0$ and $h(x)=1_A(x)$. Then, we have
$$
h(\Phi^{(n,nN)}x)=\left\{
\begin{array}{lr}
1, & \sum_{l=0}^{nN}\binom{nN}{l}x_{l+1} =1 (\bmod2),\\
0, & \sum_{l=0}^{nN}\binom{nN}{l}x_{l+1} =0 (\bmod2).\\
\end{array}
\right.
$$
By Theorem \ref{thm1}, we know that for $\vec{v}=(1,N)$,
$(\{0,1\}^{\mathbb{Z}_+},\mathcal{X},\mu,\Phi)$ is $\vec{v}$-weak
mixing. Therefore  one has, by Theorem \ref{thm3},
\begin{equation*}
\begin{split}
\lim_{p\to \infty}\frac{1}{p}\#\{n\in[0,p-1]:&\sum_{l=0}^{nN}\binom{nN}{l}x_{l+1} =0(\bmod2)\}
=\lim_{p\to
\infty}\frac{1}{p}\sum_{n=0}^{p-1}h(\Phi^{(n,nN)}x)\\=&\lim_{p\to
\infty}\frac{1}{p}\sum_{n=0}^{p-1}U_{\Phi}^{(n,nN)}h(x)=\int_{[0,1]}
h(x)d\nu(x)=1/2.
\end{split}
\end{equation*}
This finishes the proof of Theorem \ref{thm2}
\end{proof}
\begin{remark}
In particular, if $N=1$ and $k=2$, then
$\sum_{l=0}^{n}\binom{n}{l}x_{l+1}$ is the sum of the
combinatorial number, which may be applied in combinatorial
mathematics.
\end{remark}

$\mathbf{Acknowledgements}$. The first author (H. A) thanks ICTP
for providing financial support and all facilities. The first
author (H. A) was supported by the Simons Foundation and IIE. The
second author (C. Liu) was partially supported by NNSF of China
(12090012).


\end{document}